\documentclass[12pt]{article}
\usepackage[utf8]{inputenc}
\usepackage{babel}
\usepackage{amsthm,amsmath,amssymb,amsfonts,amscd}
\usepackage{mathtools}
\usepackage{hyperref}
\usepackage{xcolor}
\usepackage{subfiles}
\usepackage{comment}
\usepackage{enumitem}
\usepackage[affil-it]{authblk}

\newtheorem{theorem}{Theorem}[section]
\newtheorem{lemma}[theorem]{Lemma}
\newtheorem{prop}[theorem]{Proposition}
\theoremstyle{definition}
\newtheorem{definition}{Definition}
\newtheorem*{remark}{Remark}

\newcommand{\R}{\ensuremath{\mathbb{R}}}

\newcommand{\Lipa}[1][\gamma]{\ensuremath{\mathcal{C}^#1}} 
\newcommand{\Na}[2][\gamma]{\ensuremath{\left\|#2\right\|_{#1}}} 
\newcommand{\SNa}[2][\gamma]{\ensuremath{\left|#2\right|_{#1}}} 
\newcommand{\Np}[2][\infty]{\ensuremath{\left\|#2\right\|_{L^#1}}}
\newcommand{\NUa}[2][\gamma]{\ensuremath{\left|#2\right|_{1,#1}}} 
\newcommand{\LipUa}[1][\gamma]{\ensuremath{\Lipa[{{1,\gamma}}]}} 
\newcommand{\LPet}[1][\gamma]{\ensuremath{c^{#1}}} 
\newcommand{\LUPet}[1][\gamma]{\ensuremath{c^{1,#1}}} 
\newcommand{\ZPet}{\ensuremath{\lambda^{*}}} 
\newcommand{\Nz}[1]{\ensuremath{\left\|#1\right\|_{*}}} 
\newcommand{\SNz}[1]{\ensuremath{\left|#1\right|_{*}}} 

\title{ Continuity of the solution map of some active scalar
equations in Hölder and Zygmund spaces}

\author{Marc Magaña\thanks{E-mail address: \texttt{marc.magana@uab.cat}}}

\affil{Departament de Matemàtiques -- UAB\\ 
Facultat de Ciències, 08193 -- Bellaterra, Barcelona, Spain}

\begin{document}

\maketitle

\begin{abstract}
We prove that the solution map for a family of non-linear transport equations in $\R^n$, with a velocity field given by the convolution of the density with a kernel that is smooth away from the origin and homogeneous of degree $-(n-1)$, is continuous in both the little Hölder class and the little Zygmund class. For particular choices of the kernel, one recovers well-known equations such as the 2D Euler or the 3D quasi-geostrophic equations.
\end{abstract}

\section{Introduction}
The study of transport equations plays a crucial role in various fields of applied mathematics, particularly in fluid dynamics and mathematical physics. In this paper, we focus on the Cauchy problem for a family of non-linear non-local transport equations given by
\begin{equation}\label{eq:transportEq}\tag{T}
    \left\{
        \begin{aligned}
            &\rho_t+v\cdot \nabla\rho=0,\\
            &v(\cdot,t)=k*\rho(\cdot,t),\\
            &\rho(\cdot,0)=\rho_0,
        \end{aligned}
    \right.
\end{equation}

where $\rho=\rho(x,t):\R^n\times \R\to \R$ is a scalar quantity usually known as the \textit{density}, $v=v(x,t):\R^n\times \R\to \R^n$ is a vector field called the \textit{velocity} and $k:\R^n\setminus\{0\}\to \R^n $ is a homogeneous kernel of degree $-(n-1)$ and smooth away from the origin and relates $\rho$ and $v$ via convolution.

This particular family of equations was studied by Cantero in \cite{cantero2022cgamma} where he proved existence and uniqueness of Hölder regular solutions. Later, Cantero, Mateu, Orobitg and Verdera \cite{cantero2021regularity} proved persistence of boundary smoothness of vortex patches for the case of odd kernels.  Notably, this family encompasses several well-known equations as special cases,  for particular choices of $k$, serving as a generalization of these models. A prime example is the vorticity formulation of the Euler equation in the plane, where $k=\nabla^\perp N$ being $N(x)$ the fundamental solution of the Laplacian. Other remarkable examples are the aggregation equation in $\R^n$ for initial conditions being the characteristic function of some domain (see \cite{bertozzi2012aggregation} for the derivation), here $k=\nabla N$, and the quasigeostrophic equation in $\R^3$, which corresponds to $k=L(\nabla N)$ for $L(x)=(-x_2,x_1,0)$, see \cite{garcia2022time}.

The existing literature on these equations is vast. For the Euler equation see, for instance \cite{chemin1998perfect,constantin2001eulerian,MajdaBertozzi} and for the aggregation \cite{bertozzi2012aggregation, bertozzi2016regularity}. The main focus in these papers was on either existence and uniqueness of Hölder regular solutions or in studying persistance of boundary smoothness of patches. The same happens in \cite{cantero2022cgamma} and \cite{cantero2021regularity}. However, the question of continuity with respect to initial conditions was not explicitly addressed in these works and is rarely studied in this context, with \cite{crippa2015strong} being an interesting exception.

According to Hadamard, a Cauchy problem is locally well-posed in a Banach space $X$ if for any initial data in $X$ there exists a unique solution that persists for some $T > 0$ in the space $C([0, T), X)$ and depends continuously on the initial data. This notion of well-posedness is rather strong and may not always be suitable, as highlihted by Kato \cite{kato1983cauchy}.  Recent works have explored ill-posedness in the sense of Hadamard of the Euler equations under certain initial conditions, providing a comprehensive survey on this topic \cite[Section 1]{misiolek2018continuity}.

The authors of the above mentioned article, Misio{\l}ek and Yoneda, proved well-posedness in the sense of Hadamard for the Euler equations in 2D and 3D for velocities in the little Hölder spaces $\LUPet,$ with $0<\gamma<1$. They also provided an example for the 3D Euler equation (which is worth noting that is not a transport equation) that shows that the problem is ill-posed, in general, for  $v\in\LipUa,$ in the sense of Hadamard. In this paper, we follow and extend their work to prove well-posedness for \eqref{eq:transportEq} in the case that the initial data is in $\LPet$, for $0<\gamma<1$. It is important to point out that we are changing the focus from the velocity field to the density, and that is the reason why the space we study is $\LPet$ rather than $\LUPet$. Moreover, we also establish Hadamard well-posedness for initial data in the little Zygmund class. 

The main theorem of this paper is the following:
\begin{theorem}\label{th:enunciatGeneral}
Let $k\in C^3(\R^n\setminus\{0\};\R^n)$ a kernel homogeneous of degree $-(n-1).$ Then
\begin{enumerate}[label=(\textbf{\Alph*})]
    \item \label{th:apartatA} For $0<\gamma < 1,$ and  $\rho_0\in \LPet_c(\R^n;\R),$ there exist $T>0$ and a unique weak solution $\rho$ to the transport equation \eqref{eq:transportEq} such that the map $\rho_0\to \rho$ is continuous from $\LPet(\R^n)$ to $C([0,T),\LPet(\R^n)).$ 
    \item \label{th:apartatB} The result in \ref{th:apartatA} is also true if one changes $\LPet$ for the little Zygmund class, $\lambda^*.$
\end{enumerate} 
\end{theorem}

The outline of this paper is as follows. In Section \ref{sec:preliminaries}, we recall the definitions of Hölder spaces and present some relevant results concerning these spaces and Calderón-Zygmund operators acting on them. In Section \ref{sec:well-posedness}, we prove part \ref{th:apartatA} of Theorem \ref{th:enunciatGeneral}, that is, we prove the result for the little Hölder spaces. Subsequently in Section \ref{sec:zygmund} we extend this result to the little Zygmund class, thereby proving part \ref{th:apartatB}.

\section{Preliminaries}\label{sec:preliminaries}
First, we recall the definitions of Hölder and little Hölder spaces. 

\begin{definition}\label{def:holder}
Given $0<\gamma<1$ and $f:\R^n\to \R$ let $$\Np{f}=\sup_{x\in \R^n} |f(x)| \text{ and } \SNa{f}=\sup_{\substack{x,y\in\R^n\\ x\neq y}} \frac{|f(x)-f(y)|}{|x-y|^\gamma}.$$
We define the norm $$\Na{f}\coloneqq \Np{f}+\SNa{f}.$$
For $F(x)=(f_1(x),\dots,f_d(x)):\R^n\to\R^d$ we define $\Na{F}\coloneqq \sup_{i=1,\dots, d} \Na{f_i}$ and the Hölder space $$\Lipa (\R^n; \R^d) \coloneqq \left\{ f: \R^n\to \R^d \;|\;\Na{f}<\infty\right\}.$$
We also define $$\NUa{F}=\Np{F}+\Np{D F} + \SNa{D F},$$ and the Hölder space $\LipUa(\R^n;\R^d)$ consisting of the functions $f:\R^n\to\R^d$ such that $\NUa {f} < \infty$. 

Let $\LPet(\R^n;\R^d)$ denote the closed subspace of $\Lipa(\R^n;\R^d)$ consisting of those functions satisfying the vanishing condition 
\begin{equation} \label{eq:vanishingCond}
    \lim_{h\to 0}\sup_{\substack{x,y\in\R^n\\ 0<|x-y|<h}} \frac{|f(x)-f(y)|}{|x-y|^\gamma}=0.
\end{equation}

Likewise, let $\LUPet(\R^n;\R^d)$ denote the closed subspace of $\LipUa(\R^n;\R^d)$ consisting of those functions whose derivatives satisfy the vanishing condition. We call $\LPet$ and $\LUPet$ little Hölder spaces. It is easy to see that $c^\gamma(\R^n)$ is the clousure of $C^\infty(\R^n)$ in the usual Hölder norm of $C^\gamma(\R^n)$.
\end{definition}

Let us now recall some well known properties for elements of the Hölder spaces as well as bounds for Calderón-Zygmund operators acting
on them.
\begin{lemma}\label{lemma:propietats}
Let $f,g\in\Lipa,$ $0<\gamma<1.$ Then 
\begin{align}
    \SNa{fg}&\leq \Np{f}\SNa{g}+\SNa{f}\Np{g} \label{eq:SNprod} \\
    \Na{fg}&\leq \Na{f}\Na{g} \label{eq:SNaAlgebra}
\end{align}
If moreover $X$ is a smooth invertible transformation in $\R^n$ satisfying 
$$|\det DX(\alpha)|\geq c_1 > 0,$$
then there exists $c>0$ such that
\begin{align}
    \Na{(DX)^{-1}}&\leq c\Na{DX}^{2n-1}\\
    \NUa{X^{-1}}&\leq c \NUa{X}^{2n-1}\\
    \SNa{f\circ X}&\leq \SNa{f}\Np{DX}^\gamma \label{eq:SNaComp}\\ 
    \Na{f\circ X}&\leq \Na{f}(1+\NUa{X}^\gamma ) \label{eq:NaComp}\\
    \Na{f\circ X^{-1}}&\leq \Na{f}(1+\NUa{X}^{\gamma(2n-1)}) \label{eq:NaCompInversa}
\end{align}
\end{lemma}

\begin{lemma}\label{lemma:SIO}
    Let $k: \R^n \to \R,$ $k\in C^3(\R^n\setminus \{0 \})$ a kernel homogeneous of degree $-(n-1).$ That is, $$k(\lambda x) = \frac{1}{\lambda^{n-1}}k(x), \quad \forall x\neq 0 \text{ and }\forall \lambda >0.$$
    Let $P=\partial_ik, i=1,\dots,n.$ Set $$Tf(x)=\int_{\R^n} k(x-x')f(x')dx' \text{ and } Sf(x)= p.v. \int_{\R^n} P(x-x')f(x')dx',$$ where $p.v.$ denotes, as usual, the Cauchy principal value.

    For $0<\gamma < 1$ let $f\in \Lipa_c(\R^n;\R).$ Set $R^n\coloneqq m(supp(f))<\infty,$ being $m$ the Lebesgue measure. Then, there exists a constant $c,$ independent of $f$ and $R,$ such that 
    \begin{align}
        \Np{Tf}&\leq cR \Np{f}\\
        \Np{Sf}&\leq c\left\{\SNa{f}\varepsilon^\gamma +\max\left(1,\ln{\frac{R}{\varepsilon}}\right)\Np{f}\right\}, \quad \forall \varepsilon>0,\label{eq:SIOepsilon}\\
        \SNa{Sf}&\leq c \SNa{f}.    \label{eq:SIOsna}     
    \end{align}
\end{lemma}

The proofs of Lemmas \ref{lemma:propietats} and \ref{lemma:SIO} can be found in \cite[p. 159 - 163, Lemmas 4.1, 4.2, 4.3, 4.5 and 4.6]{MajdaBertozzi}. It is worth noting that more hypothesis on the kernels $k$ and $P$ are required on \cite{MajdaBertozzi} but they are not needed (see \cite[Remark 5]{cantero2022cgamma}). 

We also state here the following lemma about distributional derivatives of the kernels considered on this work. For its proof, refer to \cite[Lemma 8]{cantero2022cgamma}

\begin{lemma} \label{lemma:diffKernel}
Given $k=(k_1,\dots,k_n): \R^n\setminus \{0\}\to \R, k\in C^3(\R^n\setminus \{0\})$ homogeneous of degree $-(n-1),$ we have, distributionally, for $i,j\in\{1,\dots, n\}$ $$\partial_i k_j = p.v. \partial_i k_j+c_{ij}\delta_0,$$ where $c_{ij}=\int_{\partial B(0,1)}k_j(s)s_id\sigma(s)$ and $\delta_0$ is the Dirac measure at the origin.
\end{lemma}

\section{Local well-posedness in \texorpdfstring{$\LPet$}{cg}} \label{sec:well-posedness}
In this section, we establish the main result of this paper concerning the well-posedness of the Cauchy problem \eqref{eq:transportEq} for functions in the little Hölder spaces $\LPet$ with $0 < \gamma < 1$. That is, we prove part \ref{th:apartatA} of Theorem \ref{th:enunciatGeneral}.

As in the case of Euler equation, a good way to prove an existence and uniqueness result is by dealing with an, in some sense, equivalent equation rather than the one presented in Theorem \ref{th:enunciatGeneral}. To introduce this equivalent equation, we first need to define the \textit{flow map}.

Given a velocity field $v$ and a point $\alpha\in\R^n$ we define, whenever it is well-defined, the \textit{flow map} or the \textit{particle trajectory map}, $X(\alpha,\cdot):\R\to\R^n$ as the solution of the ordinary differential equation (ODE)
\begin{equation*}
    \left\{
        \begin{aligned}
            &\frac{d}{dt}X(\alpha,t)=v(X(\alpha,t),t),\\
            &X(\alpha,0)=\alpha.
        \end{aligned}
    \right.
\end{equation*}

This map indicates the position at time t of the particle that was initially at
$\alpha$ and that has moved according to the velocity field at every moment. If $\rho(\cdot,t)$ is sufficiently smooth, a straightforward computation shows that $\rho(X(\alpha,t),t)=\rho_0(\alpha).$ Thus, it can be said that $\rho$ is transported with the flow defined by $v$, which justifies calling \eqref{eq:transportEq} a transport equation.  

We can focus then on proving existence, uniqueness and regularity for the ODE for the particle trajectory map $X$, which after a change of variables is: 
$$\frac{d X}{d t}(\alpha,t)=  \int_{R^n} k(X(\alpha,t)-X(\alpha',t))\rho_0(\alpha')\det\left(DX(\alpha',t)\right) d\alpha '.$$

A standard way to prove existence and uniqueness for an ODE is to apply
Picard-Lindelöf’s theorem, which can be stated as follows:

\begin{theorem}[Picard-Lindelöf]\label{th:Picard}
    Let $O\subseteq B$ be an open subset of a Banach space $B$ and let $F:O\to B$ be a locally Lipschitz continuous mapping. Then, given $X_0\in O,$ there exists a time $T>0$ such that the ordinary differential equation $$\frac{d X}{d t}=F(X), \quad X(\cdot,t=0)=X_0,$$ has a unique (local) solution $X\in C^1((-T,T); O).$
\end{theorem}

To apply Theorem \ref{th:Picard}, we first need an equation of type $\frac{d X}{d t}=F(X)$, which we have for 
\begin{equation}\label{def:Functional}
F(X(\alpha,t)) \coloneqq \int_{R^n} k(X(\alpha,t)-X(\alpha',t))\rho_0(\alpha')\det\left(DX(\alpha',t)\right) d\alpha ' 
\end{equation}

We also need a Banach space $B$ and an open subset of $B$ such that $X$ belongs to it. We take, as in \cite{misiolek2018continuity}, $B=\LUPet(\R^n;\R^n)$ and the open set  
\begin{equation} \label{def:OpenSet}
    U_\delta = \left\{ X: R^n \to R^n : X= e+\varphi_X, \varphi_X\in\LUPet(\R^n) \text{ and } \NUa{\varphi_X}<\delta \right\}, 
\end{equation} where $e(x) = x$ and $\delta>0$ is chosen small enough so that 
\begin{equation}
    \inf_{x\in \R^n} \det DX(x)>\frac{1}{2}.
\end{equation}
Clearly, $U_\delta$ can be identified with an open ball centered at the origin in $\LUPet(\R^n;\R^n).$

We are now ready to check the hypothesis in Picard-Lindelöf's theorem. We split them into the two following propositions.

\begin{prop} \label{prop:MapsIntoCa}
Let $U_\delta$ as defined in \eqref{def:OpenSet}. Then, the functional $F$ defined in \eqref{def:Functional} maps $U_\delta$ to $\LUPet(\R^n;\R^n).$
\end{prop}
\begin{proof}
We need to first verify that $$\Np{F(X)}+\sup_{i\in\left\{1,\dots,n\right\}}\Na{\frac{\partial}{\partial \alpha_i} F(X) } < \infty.$$

If we consider the change of variables $x'=X(\alpha')$ in \eqref{def:Functional} we get, for the $j-$component \begin{equation} \label{eq:Fj}
    F_j(X)(\alpha)=\int_{\R^n} k_j(X(\alpha)-x')\rho_0(X^{-1}(x'))dx'= (k_j*\Tilde{\rho}_0)\circ X(\alpha), 
\end{equation}
where $\Tilde{\rho}_0\coloneqq \rho_0 \circ X^{-1}.$ 

Following \cite[Section 3]{cantero2022cgamma} and setting $R_0=m(supp(\rho_0)) ^\frac{1}{n},$ one gets 
\begin{equation}\label{eq:infF}
    \Np{F(X)}\leq c R_0 \Np{D X} \Np{\rho_0}, 
\end{equation}
which is bounded for $X\in \LUPet$ and $\rho_0\in \LPet_c.$ Next, we focus on the norms of the derivatives of $F(X).$ We write $\partial_i=\frac{\partial}{\partial \alpha_i.}$ By definition of the norm we have $\Na{\partial_i F(X)}=\sup_{j\in\{1,\dots, n\}}\Na{\partial_iF_j(X)}.$ Thus, we work with $\partial_iF_j(X)$ for $i,j\in\{1,\dots, n\}$. Applying the chain rule to $F_j$ in \eqref{eq:Fj}, combined with Lemma \ref{lemma:diffKernel}, yields
\begin{equation}
    \partial_iF_j(X)(\alpha)=\sum_{r=1}^n(c_{rj}\rho_0(\alpha)+ p.v. (\partial_rk_j*\Tilde{\rho}_0)(X(\alpha)))\partial_iX_r(\alpha).
\end{equation}
Clearly, hypothesis in Lemma \ref{lemma:SIO} are satisfied if $P=\partial_rk_j,$ thus we can use \eqref{eq:SIOepsilon} to bound the $L^\infty$ norm of the derivatives:
\begin{equation}\label{eq:infGradF}
\begin{split}
\Np{\partial_iF_j(X)}&\leq\sum_{r=1}^n(|c_{rj}|\Np{\rho_0}+\Np{p.v. (\partial_rk_j*\Tilde{\rho}_0)(X)})\Np{\partial_iX_r}\\
&\leq \Np{D X} \sum_{r=1}^n(|c_{rj}|\Np{\rho_0}+\Np{p.v. (\partial_rk_j*\Tilde{\rho}_0)(X)})\\
&\leq n\Np{D X}\left(C\Np{\rho_0}+c\left\{\SNa{\Tilde{\rho}_0}+\max(1,\ln{R})\Np{\Tilde{\rho}_0}\right\}\right)\\
&\leq C (\Np{\rho_0} + \SNa{\Tilde{\rho}_0})
\end{split} 
\end{equation}

where the last step uses that $R=m(supp(\rho_0\circ X^{-1}))^\frac1n$ is bounded and $\Np{\Tilde{\rho}_0}=\Np{\rho_0}.$

To bound the Hölder seminorm we first use \eqref{eq:SNprod}
\begin{equation*}
\begin{split}
\SNa{\partial_iF_j(X)}&\leq \sum_{r=1}^n\Np{c_{rj}\rho_0+p.v. (\partial_rk_j*\Tilde{\rho}_0)(X)}\SNa{\partial_iX_r}\\&+\sum_{r=1}^n\SNa{c_{rj}\rho_0+p.v. (\partial_rk_j*\Tilde{\rho}_0)(X)}\Np{\partial_iX_r} \\&= I+ II
\end{split} 
\end{equation*}

The first term is estimated as in \eqref{eq:infGradF}: 
$$I\leq C(\Np{\rho_0}+\SNa{\Tilde{\rho}_0})\SNa{D X}\leq C(\Np{\rho_0}+\SNa{\Tilde{\rho}_0})\SNa{D\varphi_X}.$$
For the second term we use first \eqref{eq:SNaComp} and then \eqref{eq:SIOsna} from Lemma \ref{lemma:SIO}:
\begin{equation*}
    \begin{split}
        II&\leq \Np{D X}\sum_{r=1}^n |c_{rj}|\SNa{\rho_0}+ \SNa{p.v. (\partial_rk_j*\Tilde{\rho}_0) }\Np{D X}^\gamma \\
        &\leq \Np{D X}\sum_{r=1}^n |c_{rj}|\SNa{\rho_0}+ c\SNa{\Tilde{\rho}_0}\Np{D X}^\gamma \\
        &\leq C (\SNa{\rho_0}(1+\Np{D \varphi_X})+\SNa{\Tilde{\rho}_0}(1+\Np{D \varphi_X})^{1+\gamma} )
    \end{split}
\end{equation*}

Combining both estimates with the fact that $$\SNa{\Tilde{\rho}_0}=\SNa{\rho_0\circ X^{-1}}\leq \SNa{\rho_0}\Np{D X^{-1}}^\gamma, $$ we get 
\begin{equation}\label{eq:SNaGradF}
    \SNa{D F(X)}\leq C_1(\Np{\rho_0}+\SNa{\rho_0})\SNa{D \varphi_X}+C_2(1+\Np{D \varphi_X})^{1+\gamma} \SNa{\rho_0}.
\end{equation}

Combining \eqref{eq:infF}, \eqref{eq:infGradF}, \eqref{eq:SNaGradF} with the fact that $X\in U_\delta$ we get 
\begin{equation}\label{eq:NUaF}
    \NUa{F(X)}\leq C(\Np{\rho_0}+\SNa{\rho_0}).
\end{equation}

Finally, to show that $F(X)$ maps $U_\delta$ to $\LUPet$ it suffices now to observe that \eqref{eq:SNaGradF} yield
$$\lim_{h\to 0}\sup_{0<|x-y|<h}\frac{|D F(X)(x)-D F(X)(y)|}{|x-y|^\gamma}=0,$$
since $\varphi_X\in \LUPet$ and $\rho_0\in\LPet.$
\end{proof}

In order to apply Picard-Lindelöf theorem one must check now that the functional $F$ is a locally Lipschitz continuous mapping. As proved on \cite[Proposition 15] {cantero2022cgamma}, this is indeed the case.  It is important to note that the open set considered here is different to the one used on the mentioned paper; however, this does not affect the proof.

Furthermore, since the kernels considered may not be divergence-free, the standard argument used for the Euler equation becomes significantly more complex in this context, requiring a delicate treatment for the estimates.

\begin{prop}\label{prop:Lip}
Let $U_\delta$ as defined in \eqref{def:OpenSet}. Then, the functional $F:U_\delta \to \LUPet$ defined in \eqref{def:Functional} is locally Lipschitz. 
\end{prop}

 The last ingredient in order to proof \ref{th:apartatA} of Theorem \ref{th:enunciatGeneral} is the following Lemma, the proof of which can be seen in \cite[Lemma 2.2]{misiolek2018continuity}. It is worth noting that the proof, as presented on the mentioned paper, relies on the fact that one is working with functions in the little Hölder space.  

\begin{lemma} \label{lemma:contInversa}
    Let $0<\alpha<1.$ Suppose that $X,Y \in U_\delta.$ Then, the functions $X,Y \to X\circ Y$ and $X\to X^{-1}$ are continuous in the  $\LipUa$ norm topology.
\end{lemma}

\begin{proof}[Proof of Theorem \ref{th:enunciatGeneral} part \ref{th:apartatA}]
    Let $B=\LUPet(\R^n;\R^n)$ and $U_\delta$ defined in \eqref{def:OpenSet}. Then, by Propositions \ref{prop:MapsIntoCa} and \ref{prop:Lip} the functional $F$ defined in \eqref{def:Functional} satisfies the hypothesis of Picard-Lindelöf's theorem and therefore there exist $T>0$ and a unique  solution $X(\cdot,t)\in \LUPet$ to the ODE $\frac{d X}{d t}(\alpha,t)= F(X).$ Since, $\rho$ is transported by the flow $X$ it follows that there exists a solution $\rho\in\LPet_c$ to \eqref{eq:transportEq}. To see that the solution is unique, assume that $\Tilde{\rho}\in\LPet_c$ and $\Tilde{v}\in\LUPet$ is also a solution to \eqref{eq:transportEq} with the same initial data. Then we can find a trajectory $\Tilde{X}(\cdot,t)$ associated to $\Tilde{v}(\cdot,t)$ such that $\Tilde{\rho}$ is transported by $\Tilde{X}.$ But by uniqueness of the flow-map, $X=\Tilde{X}$ so $\Tilde{\rho}(\cdot,t)=\rho_0(\Tilde{X}^{-1}(\cdot,t))=\rho_0(X^{-1}(\cdot,t))=\rho(\cdot,t)$ and by convoluting with the kernel we also have that $\Tilde{v}=v.$

    We now focus on the dependence of the solutions on $\rho_0.$ Note that $\rho_0$ appears as a parameter on the right hand side of the ODE $$\frac{d X}{d t}=  \int_{R^n} k(X(\alpha,t)-X(\alpha',t))\rho_0(\alpha')\det\left(DX(\alpha',t)\right) d\alpha '=F_{\rho_0}(X)(\alpha).$$ 
    Moreover, since the dependence is linear on $\rho_0$ it follows that continuity of the map $\rho_0\to F_{\rho_0}$ is an immediate consequence of the estimate in \eqref{eq:NUaF}. Applying the theorem of continuous dependence on initial data of ODEs for Banach spaces we find that there exist $T > 0$ and a unique Lagrangian flow $X\in C([0,T),U_\delta)$ which depends continuously on $\rho_0$. 
    
    Next, take $\rho_0,\omega_0\in \LPet(\R^n;\R)$ and let $X(t)$ and $Y(t)$ be the corresponding Lagrangian flows solving the Cauchy problem in $U_\delta$. Given any $\varepsilon>0$ and using he fact that smooth functions are dense in $\LPet,$ we can choose $\phi_\varepsilon$ in $C^\infty(\R^n;\R)$ such that $$\Na{\phi_\varepsilon(x)-\rho_0}<\varepsilon.$$
    
    We estimate
    \begin{equation*}
        \begin{split}
            \Na{\rho-\omega}&=\Na{\rho_0(X^{-1}(x,t))-\omega_0(Y^{-1}(x,t))}\\
            &\leq \Na{\rho_0(X^{-1}(x,t))-\phi_\varepsilon(X^{-1}(x,t))}+\Na{\phi_\varepsilon(X^{-1}(x,t))-\phi_\varepsilon(Y^{-1}(x,t))}\\&+\Na{\phi_\varepsilon(Y^{-1}(x,t))-\omega_0(Y^{-1}(x,t)}=I+II+III
        \end{split}
    \end{equation*}
    
    $I$ is clearly bounded by $\varepsilon$. For the second term we have 
    \begin{equation*}
        \begin{aligned}
            II&\leq \int_0^1 \Na{\nabla \phi_\varepsilon(rX^{-1}-(1-r)Y^{-1})(X^{-1}-Y^{-1})}dr\\&\leq C \Na{\nabla \phi_\varepsilon \circ (rX^{-1}-(1-r)Y^{-1})} \Na{X^{-1}-Y^{-1}} \\
            & \leq C \NUa{ \phi_\varepsilon} \NUa{X^{-1}(x,t)-Y^{-1}(x,t)}
        \end{aligned}
    \end{equation*}
    where on the second inequality we have used the algebra properties of Hölder functions, \eqref{eq:SNaAlgebra}, and on the third one \eqref{eq:NaComp}. Now it is clear that $II$ converges to zero by continuity of the inversion map, see Lemma \ref{lemma:contInversa}, and the fact that $X$ converges to $Y$ in $\LUPet$ whenever $\rho_0$ converges to $\omega_0$ since the Lagrangian flows depend continuously on the initial vorticities. Finally, using \eqref{eq:NaCompInversa} we get
    \begin{equation*}
        \begin{split}
            III&\leq \Na{\phi_\varepsilon(Y^{-1}(x,t))-\rho_0(Y^{-1}(x,t)} +\Na{\rho_0(Y^{-1}(x,t))-\omega_0(Y^{-1}(x,t))}\\
            &\leq \Na{\phi_\varepsilon(Y^{-1}(x,t))-\rho_0(Y^{-1}(x,t)} +\Na{\rho_0-\omega_0}(1+\NUa{Y}^{\gamma (2n-1)})
        \end{split}
    \end{equation*}
    
    where the first term is bounded by $\varepsilon$ and the second term converges to 0 whenever $\rho_0$ converges to $\omega_0$.
\end{proof}

\section{Local well-posedness in the little Zygmund class, \texorpdfstring{$\lambda^*$}{lambda}} \label{sec:zygmund}
In this section, we address the well-posedness of the Cauchy problem \eqref{eq:transportEq} for functions in the little Zygmund class. For this case, the existence and uniqueness of solutions are also novel contributions. 

The Zygmund class was introduced by Zygmund in the 1940s when he observed that the conjugate
function of a Lipschitz function in the unit circle does not need to be Lipschitz but it is in
the Zygmund class \cite{Zygmund-1945}. 

\begin{definition}\label{def:Zyg}
    We define the Zygmund class, $\Lambda^*(\R^n;\R)$, consisting of all bounded continuous functions $f:\R^n\to\R$ for which $$\SNz{f}=\sup_{\substack{x,h\in\R^n\\ h\neq 0}}\frac{|f(x+h)-2f(x)+f(x-h)|}{|h|}<\infty,$$
    We equip $\Lambda^*$ with the norm $\Nz{f}\coloneqq \Np{f}+\SNz{f}<\infty$. 

    As in the definition of Hölder spaces, see definition \ref{def:holder}, one can extend the definition of the Zygmund class to functions $F:\R^n\to\R^d$. 

    We also define the little Zygmund class, $\ZPet(\R^n;\R^d)$ as the closed subspace of $\Lambda^*(\R^n;\R^d)$ consisting of those functions satisfying the vanishing condition $$\lim_{\delta\to 0} \sup_{x\in\R^n}\sup_{|h|<\delta} \frac{|f(x+h)-2f(x)+f(x-h)|}{|h|}=0,$$
    and as noted for the little Hölder spaces one can see that $\lambda^*(\R^n)$ is the clousure of $C^\infty(\R^n)$ in the usual Zygmund norm of $\Lambda^*(\R^n)$.
\end{definition}

It is well known that $\operatorname{Lip}(\R^n)\subset \Lambda^*(\R^n) \subset \Lipa(\R^n)$ for any $0<\gamma <1$ and, in fact, the Zygmund class is the natural substitute of the Lipschitz class in many different contexts. 

We now present a lemma that will be crucial later.

\begin{lemma}\label{lemma:Zyg}
    Let $f,g\in\Lambda^*$ and let $X\in \LipUa$, for some $0<\gamma<1,$ be an invertible transformation with $|\det \nabla X|\geq c_1 >0.$ Then 
    \begin{align}
        \Nz{fg}&\leq c \Nz{f}\Nz{g}, \label{eq:Zalgebra}\\
        \SNz{f\circ X}&\leq c \SNz{f}(1+ \NUa{X}). \label{eq:comp}
    \end{align}    
\end{lemma}

\begin{proof}
    The fact that $\Lambda^*$ is an algebra, and therefore that the first inequality holds, is well known, see, for instance \cite[Chapter 4.6, Theorem 2]{RunstSickel1996}. However, we give here an alternative proof that was shown to us by J.J. Donaire.  

    Notice that 
    \begin{multline*}
        f(x+h)g(x+h)-2f(x)g(x)+f(x-h)g(x-h)=\\(f(x+h)-2f(x)+f(x-h))(g(x+h)-2g(x)+g(x-h))\\+f(x)(g(x+h)-2g(x)+g(x-h))+ g(x)(f(x+h)-2f(x)+f(x-h))\\-(f(x+h)-f(x))(g(x-h)-g(x))-(f(x-h)-f(x))(g(x+h)-g(x))
    \end{multline*}

    Thus, using the definition of the seminorm and the fact that Zygmund functions are lip-log we get \begin{multline*}
        |f(x+h)g(x+h)-2f(x)g(x)+f(x-h)g(x-h)|\leq\\ |h|(\SNz{f}\Np{g}+\Np{f}\SNz{g}) + c|h|^2\SNz{f}\SNz{g}\left(1+\log^2\frac{1}{|h|}\right)
    \end{multline*}
    which implies the first inequality since $\Np{fg}\leq \Np{f}\Np{g}.$ 
    
    We now prove the second inequality for small $h$, since otherwise one can just use the $L^\infty$ norm for the bounds. Notice that if we set $k=X(\alpha-h)-X(\alpha)$ we have
    \begin{multline*}
        |f(X(\alpha+h))-2f(X(\alpha))+f(X(\alpha-h))|\leq\\ |f(X(\alpha)+k)-2f(X(\alpha))+f(X(\alpha)-k)|+|f(X(\alpha+h)-f(X(\alpha)-k)| = I +II         
    \end{multline*}

    It is clear that $$I\leq \SNz{f}|k|=\SNz{f}|X(\alpha-h)-X(\alpha)|\leq \SNz{f} \Np{\nabla X} |h|$$
    and for $II$ we use again that Zygmund functions are lip-log 
    \begin{equation*}
        \begin{aligned}
        II&\leq C\SNz{f}|X(\alpha+h)-X(\alpha)+k|\log\frac{1}{|X(\alpha+h)-X(\alpha)+k|}\\
        &=C\SNz{f}|X(\alpha+h)-2X(\alpha)+X(\alpha-h)|\log\frac{1}{|X(\alpha+h)-2X(\alpha)+X(\alpha-h)|}\\ 
        &\leq C\SNz{f}|h|^{1+\gamma}\SNa{\nabla X}\log\frac{1}{\SNa{\nabla X}|h|^{1+\gamma}}\\
        &= C\SNz{f}|h||h|^{\gamma}\SNa{\nabla X}\left(\log\frac{1}{\SNa{\nabla X}|h|^{\gamma}}+\log \frac{1}{h}\right) \\
        &\leq C\SNz{f}|h|(1+\SNa{\nabla X})  
        \end{aligned}
    \end{equation*}
    where in the second inequality we use that $t\log \frac{1}{t}$ is increasing in $\left(0,\frac{1}{e}\right)$, together with the mean value theorem inequality and the that $X\in\LipUa$. The final inequality holds since $h$ is small enough, so that $|h|^\gamma\SNa{DX}\leq \frac{1}{e},$ and therefore, both terms  $|h|^{\gamma}\SNa{\nabla X}\log\frac{1}{\SNa{\nabla X}|h|^{\gamma}}$ and $|h|^{\gamma}\log \frac{1}{h}$ are bounded. Putting all together and using the definition of  $\NUa{X}$ we get the desired inequality.
\end{proof}
\begin{remark}
    Notice that if $f\in\ZPet$ and $X\in\LUPet$ for some $0<\gamma<1$ the proof of \eqref{eq:comp} also gives that $f\circ X \in \ZPet.$
\end{remark}


We finally prove \ref{th:apartatB}. Consider $\rho_0\in\ZPet.$ Since $\ZPet\subset \LPet$ for all $0<\gamma<1$ by the proof of Theorem \ref{th:enunciatGeneral} part \ref{th:apartatA} there exists a unique flow $X(\cdot,t)\in \LUPet$ for all $0<\gamma<1$, and also $X^{-1}(\cdot,t)\in \LUPet$. Since the solution is transported by the flow it follows by \eqref{eq:comp} from the previous lemma that the solution $\rho(\cdot,t)=\rho_0(X^{-1}(\cdot,t))$ is in $\ZPet.$ To verify that the solution depends continuously on the initial data, one can follow the argument used for the Hölder class, replacing the norm $\Na{\cdot}$ with $\Nz{\cdot}$. Notice that by Lemma \ref{lemma:Zyg} one can prove analogs of Lemmas \ref{lemma:propietats} and \ref{lemma:SIO} but for the Zygmund class. The only exception being \eqref{eq:SIOsna}, whose proof is more intricate and can be found in \cite[Corollary 2.4.2.]{grafakos2009modern}. With these results, and the space inclusion $\LUPet\subset \ZPet \subset \LPet,$ the final part of the proof of \ref{th:apartatA} can be replicated.

\bibliographystyle{plain} 
\bibliography{biblio}  

\begin{thebibliography}{10}

\bibitem{bertozzi2016regularity}
A.~Bertozzi, J.~Garnett, T.~Laurent, and J.~Verdera.
\newblock The regularity of the boundary of a multidimensional aggregation patch.
\newblock {\em SIAM Journal on Mathematical Analysis}, 48(6):3789--3819, 2016.

\bibitem{bertozzi2012aggregation}
A.~Bertozzi, T.~Laurent, and F.~L{\'e}ger.
\newblock Aggregation and spreading via the newtonian potential: the dynamics of patch solutions.
\newblock {\em Mathematical Models and Methods in Applied Sciences}, 22(supp01):1140005, 2012.

\bibitem{cantero2022cgamma}
J.~C. Cantero.
\newblock ${C}^{\gamma}$ well-posedness of some non-linear transport equations.
\newblock {\em Indiana University Mathematics Journal}, 73(1):81--109, 2024.

\bibitem{cantero2021regularity}
J.~C. Cantero, J.~Mateu, J.~Orobitg, and J.~Verdera.
\newblock The regularity of the boundary of vortex patches for some nonlinear transport equations.
\newblock {\em Analysis \& PDE}, 16(7):1621--1650, 2023.

\bibitem{chemin1998perfect}
J.Y. Chemin.
\newblock {\em Perfect Incompressible Fluids}.
\newblock Oxford lecture series in mathematics and its applications. Clarendon Press, 1998.

\bibitem{constantin2001eulerian}
P.~Constantin.
\newblock An eulerian-lagrangian approach for incompressible fluids: local theory.
\newblock {\em Journal of the American Mathematical Society}, 14(2):263--278, 2001.

\bibitem{crippa2015strong}
G.~Crippa, E.~Semenova, and S.~Spirito.
\newblock Strong continuity for the 2d euler equations.
\newblock {\em Kinetic and Related Models}, 8(4):685--689, 2015.

\bibitem{garcia2022time}
C.~Garc{\'\i}a, T.~Hmidi, and J.~Mateu.
\newblock Time periodic solutions for 3d quasi-geostrophic model.
\newblock {\em Communications in Mathematical Physics}, 390(2):617--756, 2022.

\bibitem{grafakos2009modern}
L.~Grafakos.
\newblock {\em Modern fourier analysis}, volume 250.
\newblock Springer, 2009.

\bibitem{kato1983cauchy}
T.~Kato.
\newblock On the cauchy problem for the (generalized) korteweg-de vries equation.
\newblock {\em Studies in Appl. Math. Ad. in Math. Suppl. Stud.}, 8:93--128, 1983.

\bibitem{MajdaBertozzi}
A.J. Majda and A.~Bertozzi.
\newblock {\em {Vorticity and Incompressible Flow. }}.
\newblock Cambridge Texts in Applied Mathematics. University of Cambridge, 2002.

\bibitem{misiolek2018continuity}
G.~Misio{\l}ek and T.~Yoneda.
\newblock Continuity of the solution map of the euler equations in h{\"o}lder spaces and weak norm inflation in besov spaces.
\newblock {\em Transactions of the American Mathematical Society}, 370(7):4709--4730, 2018.

\bibitem{RunstSickel1996}
T.~Runst and W.~Sickel.
\newblock {\em Sobolev Spaces of Fractional Order, Nemytskij Operators, and Nonlinear Partial Differential Equations}.
\newblock De Gruyter, Berlin, New York, 1996.

\bibitem{Zygmund-1945}
A.~Zygmund.
\newblock Smooth functions.
\newblock {\em Duke Mathematical Journal 1945-mar vol. 12 iss. 1}, 12:47--76, mar 1945.

\end{thebibliography}

\end{document}